\documentclass[a4paper]{article}
\usepackage{graphicx} 
\usepackage{amsmath, amsthm, amsfonts, amssymb, amscd, bm, enumerate}
\usepackage{verbatim}

\usepackage{url}

\newtheorem{theorem}{Theorem}
\newtheorem{corollary}{Corollary}
\newtheorem{lemma}{Lemma}
\newtheorem{definition}{Definition}
\newtheorem{assumption}{Assumption}

\theoremstyle{remark}
\newtheorem{remark}{Remark}

\def\F{\mathcal{F}}

\def\G{\mathcal{G}}

\def\X{\mathcal{X}}

\def\bR{\mathbb{R}}
\def\bN{\mathbb{N}}

\def\bone{\mathbf{1}}
\def\bP{\mathbb{P}}
\def\bQ{\mathbb{Q}}

\def\diag{\mathrm{diag}}
\mathchardef\mhyphen="2D
\def\tloc{{t\mhyphen\mathrm{loc}}}
\def\esssup{\mathop{\mathrm{ess\,sup}}}
\def\essinf{\mathop{\mathrm{ess\,inf}}}
\def\argmax{\mathop{\mathrm{arg\,max}}}
\begin{document}
\title{Undiscounted Markov chain BSDEs to stopping times}
\author{Samuel N. Cohen\thanks{Thanks to Shige Peng, Gechun Liang, Jeff Dewynne and Chris Lustri for useful conversations, in particular to Chris Lustri for discussions on circuits violating Ohm's law. Thanks also to research support from the Oxford--Man Institute for Quanititative Finance.}\\ Mathematical Institute\\ University of Oxford \\}
\date{\today}

\maketitle

\begin{abstract}
We consider Backward Stochastic Differential Equations in a setting where noise is generated by a countable state, continuous time Markov chain, and the terminal value is prescribed at a stopping time. We show that, given sufficient integrability of the stopping time and a growth bound on the terminal value and BSDE driver, these equations admit unique solutions satisfying the same growth bound (up to multiplication by a constant). This holds without assuming that the driver is monotone in $y$, that is, our results do not require that the terminal value be discounted at some uniform rate. We show that the conditions are satisfied for hitting times of states of the chain, and hence present some novel applications of the theory of these BSDEs.
\end{abstract}

\section{Introduction}
Since their early introduction by Bismut \cite{Bismut1973} (in the linear case), and in particular since the nonlinear existence result of Pardoux and Peng \cite{Pardoux1990}, Backward Stochastic Differential Equations have risen to be a powerful component of the stochastic analyst's toolkit. The early conditions of the theory, that the driver of the BSDE be Lipschitz, that the terminal time is deterministic, that the filtration is generated by a Brownian motion, have all been extended, leaving a rich class of equations which can be used in many practical problems. In particular for our discussion here, the theory of BSDEs when noise is driven by a finite or countable state Markov chain (in continuous time) has been developed by the author and collaborators in a series of recent papers (\cite{Cohen2008}, \cite{Cohen2009}, \cite{Cohen2011a}, \cite{Cohen2012}).

In this paper, we consider a novel extension of the theory of BSDEs, when the terminal time is replaced by an unbounded stopping time. We say this is a novel extension, however previous work has approached this question in the Brownian setting, in particular Darling and Pardoux \cite{Darling1997}, Briand and Hu \cite{Briand1998} and Royer \cite{Royer2004}. In those papers, a key assumption of monotonicity of the driver is made, essentially corresponding to discounting the future at a rate which is bounded away from zero. This allows these authors to show that the problem is well-posed, in particular, that the BSDE admits unique solutions and that these solutions can be well approximated using finite-time BSDEs. A related setting is discussed in \cite[Section 2]{Cohen2012} in the Markov chain case, again with an assumption of a discounting term. 

In this paper, we approach the problem somewhat differently. We seek to show that one can impose conditions on the stopping time directly such that the BSDE admits unique solutions (at least, unique solutions satisfying some integrability conditions), and that these conditions can be verified, for example, when the stopping time is a hitting time. In this case, our conditions can be connected to the uniform ergodicity of the Markov chain under a family of measures. These questions of ergodicity were explored in \cite{Cohen2012} in the context of proving the existence of Ergodic BSDE solutions, and we draw liberally on these results. In all our analysis, we do not assume the monotonicity of the driver, which in particular allows us to consider the case when the driver depends only on the $Z$ component of the solution, that is, when the BSDE is effectively performing a nonlinear change of measure.

We then present a couple of applications, both the standard application of BSDEs to control problems, and a novel application to our understanding of non-Ohmic electronic circuits, by relating these to a Markov chain model.

The paper first presents the basic theory of BSDEs on Markov chains (Section \ref{basic}), then gives our existence proof (Section \ref{STBSDE}), then considers how the conditions of the existence proof could be verified for hitting times (Section \ref{hitting}), and ends with applications (Section \ref{applic}).

\subsection{Introducing BSDEs on Markov chains}\label{basic}

Consider a continuous-time countable-state process $X$ in a probability space $(\Omega, \F, \bP)$. We shall suppose that $X$ is a Markov chain in its own filtration, under a measure $\bP$.  Without loss of generality, we shall represent $X$ as taking values from the standard basis vectors $e_i$ of $\bR^N$ (where $N\in\bN\cup\{\infty\}$ is the number of states, and $\bR^\infty$ denotes the space of infinite real sequences). We write $\X$ for this set of basis vectors. For notational simplicity, we will think of all vectors as column vectors, and denote by $z^*$ the transpose of $z$ (so that $z^* y$ is the Euclidean or $\ell_2$ inner product, and $e_i^*z$ is the $i$th component of $z$). An element $\omega\in\Omega$ can be thought of as describing a path of the chain $X$.

Let $\{\F_t\}_{t\geq 0}$ be the completion of the filtration generated by $X$, that is,
\[\F_t = \sigma(\{X_s\}_{s\leq t}) \vee \{B\in\F_\infty:\bP(B)=0\}.\]
 As $X$ is a right-continuous pure jump process which does not jump at time $0$, this filtration is right-continuous, and we assume $\F=\F_\infty=\bigvee_{t<\infty}\F_t$.  For the basic theory of continuous-time countable-state Markov chains, see for example Rogers and Williams \cite[Vol. 1, p228ff]{Rogers2000}, for the approach taken here, see Elliott, Aggoun and Moore \cite[Part III]{Elliott1994a}).

Let $A$ denote the possibly infinite rate matrix\footnote{In our notation, as in \cite{Elliott1994a}, $A$ is the matrix with entries $A_{ij}$, where $A_{ij}$ is the rate of jumping from state $j$ to state $i$. Depending on the convention used, this is either the rate matrix or its transpose. In our notation $A^*$, the transpose of $A$, is the generator of the Markov chain.} of the chain $X$. Note that $(A_t)_{ij}\geq 0$ for $i\neq j$ and $\sum_i (A_t)_{ij} = 0$ for all $j$ (the columns of $A$ all sum to $0$). We assume, for simplicity, that the entries in $A_t$ are uniformly bounded, and so the chain is regular.

From the Doob--Meyer decomposition (see \cite[Appendix B]{Elliott1994a}), we write our chain in the following way
\begin{equation}\label{ref:DMrepMarkovChain}
X_t = X_0 + \int_{]0,t]} A_u X_{u-} du + M_t,
\end{equation}
where $M$ is a locally-finite-variation pure-jump $\bP$-martingale in $\bR^N$, and the chain starts in state $X_0\in\bR^N$. Our aim is to study a class of BSDEs up to stopping times, that is, equations of the form
\begin{equation}\label{eq:STBSDE}
 Y_t = \xi+\int_{]t,\tau]} f(\omega, u, Y_{u-}, Z_u) du - \int_{]t,\tau]} Z_u^* dM_u, \qquad 0\leq t\leq \tau< \infty,
\end{equation}
where $\tau$ is an integrable stopping time, $f:\Omega\times\bR^+\times \bR\times \bR^N \to \bR$ is a given function, $Y$ is a real-valued c\`adl\`ag stochastic process, $Z$ is a predictable process in $\bR^N$ such that
\[\int_{]0,t]} Z_u^*dM_u := \sum_i \int_{]0,t]} (Z_u)^i d(M)^i_u\]
is a square-integrable martingale, (here $(\cdot)^i$ denotes the $i$th component of the vector), and $\xi$ is $\F_{\tau}$-measurable. The key problem is that we do not assume that $\tau$ is uniformly bounded, nor that $f$ has monotone dependence on $Y$. Such assumptions are needed to apply the `discounted' BSDE methods in \cite{Cohen2012} or \cite{Briand1998}.

\begin{remark}
 We note that the use of left limits for $Y$ in the driver term of (\ref{eq:STBSDE}) initially seems unconventional, for those used to the theory of BSDEs in a Brownian setting. However, it is the natural approach when the driver term can itself jump (see \cite{Cohen2010}), it ensures that the driver is predictable, and as the integral is with respect to Lebesgue measure and our processes have at most countably many jumps, in this case the equation is unchanged whether the left limits are included or not.
\end{remark}

Of importance will be the following process and the associated spaces.
\begin{definition}
 Let
\[\psi_t:= \diag(A_t X_{t-}) - A_t \diag(X_{t-}) - \diag(X_{t-}) A_t^*.
\]
Then $\psi$ is a predictable process taking values in the symmetric, positive semidefinite matrices in $\bR^{N\times N}$, with the property that
\[E\Big[\Big(\int_{]0,t]} Z_u^*dM_u\Big)^2\Big] = \int_{]0,t]} E[Z^*_u \psi_u Z_u] du\]
for any $t$ and any predictable processes $Z$ of correct dimension (see \cite{Cohen2008}). For simplicity, we write
\[\|z\|^2_{M_t} := z^*\psi_t z,\]
and note that this is a stochastic seminorm.

We define the following spaces of processes.
\begin{itemize}
 \item $Y\in S^2$ if $E[\sup_{t\in\bR^+} Y^2_t]<\infty$ and $Y$ is c\`adl\`ag,
 \item $Z\in H^2_M$ if $E\big[\int_{]0,\infty]} \|Z_t\|_{M_t}^2 dt\big] < \infty$ and $Z$ is predictable,
 \item $Z\sim_M Z'$ if $\|Z_t-Z'_t\|_{M_t}=0$ for almost all $t$.
 \item $Y\in S^2_{\tloc}$ if $\{Y_tI_{t\leq T}\}\in S^2$ for any $T<\infty$, and similarly $H^2_{M, \tloc}$.
\end{itemize}
We note that the definition of $S^2_{\tloc}$ is not the same as the set of processes locally in $S^2$ as usually understood in stochastic analysis, as the requirement is for any deterministic $T$, rather than for a sequence of stopping times.
\end{definition}

The basic existence theorem for BSDEs in this setting is the following.
\begin{theorem}\label{thm:basicBSDEexist}
Let $T$ be a finite deterministic time, and $f:\Omega\times[0,T]\times\bR\times\bR^N \to\bR$ be a predictable function. If $f$ is uniformly Lipschitz in $y$ and $z$, that is, there exists a constant $c>0$ such that
\[|f(\omega, t, y, z)- f(\omega,t,y', z')|^2 \leq c(|y-y'|^2 + \|z-z'\|^2_{M_t}),\]
and
\[E\Big[ \int_{]0,T]} |f(\omega, t, 0, 0)|^2 dt \Big] < \infty\]
then for any $\xi\in L^2(\F_T)$, there exists a unique solution $(Y,Z) \in S^2 \times H^2_M$ to the BSDE
\[\xi=Y_t- \int_{]t,T]} f(\omega, u, Y_u, Z_u) du + \int_{]t,T]} Z_u^* dM_u.\]
\end{theorem}

\begin{proof}
For the finite state case, this result is given in \cite{Cohen2008}. For the infinite state case, we use the martingale representation result established in \cite{Cohen2008},  which naturally extends to general spaces, coupled with the existence result for BSDEs in general spaces established in \cite{Cohen2010}.
\end{proof}

We recall a key definition from \cite{Cohen2012}.
\begin{definition}\label{control}
 Consider $A$ and $B$ (possibly infinite) rate matrices, that is, matrices with $A_{ij}\geq 0$ for $i\neq j$ and $\sum_i A_{ij} = 0$ for all $j$, and similarly for $B$. We write $E^A$ for the expectation under the measure where $X$ is a Markov chain with rate matrix $A$.

For $\gamma>0$, we shall say that $B$ is $\gamma$-\emph{controlled} by $A$ whenever $B-\gamma A$ is also a rate matrix, and the diagonal entries of $B-\gamma A$ are at most $-\gamma$. If $B$ is $\gamma$-controlled by $A$, then we shall write $A\preceq_\gamma B$. 

If $A\preceq_\gamma B$ and $B\preceq_\gamma A$, we shall write $A\sim_\gamma B$.
\end{definition}

A key result in the analysis of BSDEs is the comparison theorem. In the case of BSDEs with Markov chain noise, and in general for BSDEs with jumps, a further condition is required to ensure that the result holds. In \cite{Cohen2009, Cohen2011a}  a general condition under which the comparison theorem holds is presented, and in \cite{Cohen2008b} a condition specific to finite state Markov chain BSDEs was also given. We here give another variant, which in some sense underlies the others.

\begin{definition}\label{defn:balanced}
 For a driver $f$, we say that $f$ is \emph{$\gamma$-balanced} if there exists a random field $\lambda:\Omega\times\bR^+\times\bR^N\times\bR^N\to\bR^N$, with
$\lambda(\cdot,\cdot, z,z')$  predictable and $\lambda(\omega, t, \cdot, \cdot)$ Borel measurable, such that
\begin{itemize}
 \item $f(\omega, t, y, z)-f(\omega, t, y,z')= (z-z')^*(\lambda(\omega, t, z,z') -AX_{t-}),$
 \item for each $e_i\in\X$,
\[\frac{e_i^*\lambda(\omega, t, z, z')}{e_i^* A X_{t-}} \in [\gamma, \gamma^{-1}]\]
for some $\gamma>0$, where $0/0:=1$,
 \item $\bone^*\lambda(\omega, t, z,z')\equiv0$, for $\bone\in\bR^N$ the vector with all entries $1$ and
 \item  $\lambda(\omega, t, z+\alpha\bone, z') = \lambda(\omega, t, z, z')$ for all $\alpha\in\bR$.
\end{itemize}
For simplicity, we write $\lambda^{z,z'}_t$ for $\lambda(\omega, t, z, z')$.
 
\end{definition}
\begin{remark}\label{rem:measchangebalanced}
 A minor variation on the proof of \cite[Lemma 12]{Cohen2012} shows that for $f$ to be $\gamma$-balanced it is sufficient, but not necessary, that for all $y, z, z'$,
\[\frac{f(\omega, t, y, z)- f(\omega,t,y,z')}{\|z-z'\|^2_{M_t}}(z-z')^* \Delta M_t >-1+\gamma\]
up to indistinguishability. A significant case where this condition may not hold is when $f(\omega, t, y, z) = z^*(B-A)X_{t-}$, for some $B\sim_\gamma A$. Here we see directly that $f$ is $\gamma$-balanced with $\lambda^{z,z'}_t = BX_{t-}$.
\end{remark}

\begin{lemma}\label{lem:balislip}
 If $f$ is $\gamma$-balanced, then it is Lipschitz with respect to $z$ under the $\|\cdot\|_{M_t}$-seminorm.
\end{lemma}
\begin{proof}

For each $v\in\bR^N$, we can write,
\[\|v\|^2_{M_t} = v^*\psi_tv = \sum_{e_i\neq X_{t-}} (e_i^*v_i-X_{t-}^*v)^2 (e_i^* A X_{t-}).\]
Therefore, 
\begin{equation}\label{eq:vAXLip}
 \begin{split}
  (v^*AX_{t-})^2 &= \sum_{\{i,j:e_i, e_j\neq X_{t-}\}}(e_i^*v_i-X_{t-}^*v)(e_j^*v_i-X_{t-}^*v) (e_i^* A X_{t-})(e_j^* A X_{t-})\\
&\leq  \sum_{\{i:e_i\neq X_{t-}\}}(e_i^*v_i-X_{t-}^*v)^2 (e_i^* A X_{t-})\sum_{\{j:e_j\neq X_{t-}\}}(e_j^* A X_{t-})\\
&= |X_{t-}AX_{t-}|\, \|v\|^2_M.
 \end{split}
\end{equation}
As we assumed $|X_{t-}AX_{t-}|$ is bounded, this shows that $v\mapsto v^*AX_{t-}$ is Lipschitz in the $\|\cdot\|_{M_t}$ seminorm. By assumption, we have
\[
 f(\omega, t, y, z)-f(\omega, t, y,z')= (z-z')^*(\lambda^{z,z'}_t -AX_{t-}),
\]
so, as  $\lambda^{z,z'}_t = D A X_{t-}$, for $D$ some diagonal matrix with diagonal entries in $[\gamma, \gamma^{-1}]$, we know
 \[f(\omega, t, y, z)-f(\omega, t, y,z')= (z-z')^*(D-I)(A X_{t-}) = ((D-I)(z-z'))^*(A X_{t-}).\]
Therefore, by (\ref{eq:vAXLip}),
\[ (f(\omega, t, y, z)-f(\omega, t, y,z'))^2 \leq   |X_{t-}AX_{t-}|\,\|(D-I)(z-z')\|^2_{M_t}.\]
 As $f$ is $\gamma$-balanced we know $f(\omega, t, y, z)=f(\omega, t, y, z+\alpha\bone)$ for any $\alpha$, so without loss of generality we can write $z-z'$ in a form such that $X_{t-}^*(z-z')\equiv 0$. For each $e_i$ we have $(D-I)e_i = \kappa_i e_i$ for some $\kappa_i \in [\gamma-1, \gamma^{-1}-1]$. As $\gamma<1$, for these $z-z' =: \sum_{\{i:e_i\neq X_{t-}\}} v_i e_i$,
\begin{align*}\|(D-I)(z-z')\|^2_{M_t} &= \Big\|\sum_{\{i:e_i\neq X_{t-}\}}\kappa_i v_i e_i\Big\|^2_{M_t} = \sum_{\{i:e_i\neq X_{t-}\}} (\kappa_i v_i)^2 (e_i^*AX_{t-}) \\
& \leq \gamma^{-1}\sum_{\{i:e_i\neq X_{t-}\}} v_i^2 (e_i^*AX_{t-}) = \gamma^{-1}\|z-z'\|^2_{M_t},
\end{align*}
which yields the result.

\end{proof}

In the following lemma, we use the notion of the essential supremum of a family of random variables (as opposed to the essential supremum of a single random variable). This is similar conceptually to the supremum taken pointwise for each $\omega$, with more care taken to ensure measurability and to prevent sets of measure zero contaminating the result. A construction and discussion of this concept can be found in \cite[Appendix A5]{Follmer2002}.
\begin{lemma}\label{lem:gambalancedconvex}
Let $\{f(u; \omega, t, y, z)\}_{u\in U}$ be a family of $\gamma$-balanced drivers. Then 
\[g(\omega, t, y, z) := \esssup_{u\in U}\{f(u; \omega, t, y, z)\}\]
 is also $\gamma$-balanced, and similarly for $\essinf_{u\in U} f$.
\end{lemma}
\begin{proof}
Suppose first that for almost all $\omega$, all $t,y,z$ the supremum is attained. Then, omitting $\omega, t, y$ for notational simplicity, there exists some predictable control $u^*$ such that
\[g(z) - g(\bar z) \leq f(u^*;  z) - f(u^*;  \bar z)= (z-\bar z) (\lambda^{*}-AX_{t-})\]
for some $\lambda^{*}$ satisfying the requirements of Definition \ref{defn:balanced} for $f(u^*;\cdots)$.  Similarly, there exists some $u_*, \lambda_*$ such that 
\[g(z) - g(\bar z) \geq f(u_*;  z) - f(u_*; \bar z) = (z-\bar z) (\lambda_*-AX_{t-})-\epsilon.\]
As $g( z) - g(\bar z)$ is scalar, we see that for 
\[\alpha:= \frac{g( z) - g(\bar z) - f(u_*; z) + f(u_*; \bar z)}{f(u^*; z) - f(u^*;\bar z)- f(u_*; z) + f(u_*; \bar z)}\]
we have
\[g(z) - g(\bar z) = (z-\bar z) (\alpha \lambda^{*}+(1-\alpha)\lambda_*-AX_{t-}).\]
Finally, we note that $\lambda^{z,\bar z}_t:=\alpha \lambda^{*}+(1-\alpha)\lambda_*$ satisifes all the requirements of Definition \ref{defn:balanced}.

If the supremum is not attainable, then it is possible (but tedious) to construct an appropriate approximation sequence and to show that the corresponding vectors $\lambda$ have a convergent subsequence in $\ell_2$, using the boundedness properties of $\lambda$. The limit of this sequence will then satisfy the requirements of the theorem. The details are left to the reader.
\end{proof}

\begin{theorem}[Finite-time comparison theorem]\label{thm:finitecompthm}
Let $(Y, Z)$ and $(Y', Z')$ be the solutions to two BSDEs with drivers $f$ and $f'$. Suppose $f$ is $\gamma$-balanced and $f(\omega, t, y, z)\geq f'(\omega, t, y, z)$ for all $(y,z)$, $dt\times d\mathbb{P}$-a.s.  Then $Y_T\geq Y'_T$ a.s. implies $Y_t\geq Y'_t$ a.s. up to indistinguishability.
\end{theorem}

 The finite-time comparison theorem is easy to deduce (see for example \cite{Cohen2009}) from the following lemma.
\begin{lemma}\label{lem:girsanovvalid}
 If $f$ is $\gamma$-balanced, then for any predictable processes $Z,Z'\in H^2_{M,\tloc}$, any process $Y\in S^2_{\tloc}$, any $T<\infty$, there exists a probability measure $\bQ^T$ equivalent to $\bP$ such that
\[\tilde M_t=\int_{]0,t\wedge T]}\big(f(\omega, s, Y_{s-}, Z_s)- f(\omega,s,Y_{s-},Z'_s) \big)ds + \int_{]0,t\wedge T]}(Z_s-Z'_s)^*dM_s\]
is a $\bQ^T$-martingale.
\end{lemma}
\begin{proof}
Let $\bQ^T$ be the measure under which $X$ jumps, at time $t\in[0,T]$, to a state $e_i\neq X_{t-}$ at a rate $e_i^*\lambda^{Z_t,Z_t'}_t$, where $\lambda(\cdots)$ is the random field associated with $f$ by Definition \ref{defn:balanced}. This is a predictable bounded process (as $A$ is bounded), and so the measure $\bQ^T$ is well defined. Under this measure, $X$ has a semimartingale decomposition
\[X_t=X_0 + \int_{]0,t]} \lambda^{Z_u,Z_u'}_u du + \hat M_t\]
for $\hat M_t$ a $\bQ^T$ martingale. We can then verify that $\tilde M_t = \int_{]0,t]}(Z_u-Z_{u-})^*d\hat M_u$, and a calculation of the predictable $\bQ^T$-quadratic variation of $\tilde M_t$ guarantees the desired martingale property. As the relative rates of $X$'s jumps under $\bP$ and $\bQ$ are bounded, by assumption on $\lambda$, we can deduce that $\bQ^T$ and $\bP$ are indeed equivalent measures.
\end{proof}

The following lemma extends this measure to all time horizons, and expresses its properties in a useful manner for the study of BSDEs.

\begin{lemma}\label{lem:YisQmart}
Let $Y$ and $\bar Y$ be two processes in $S^2_\tloc$ with dynamics
\begin{align*}
 dY_t &= -f(\omega, t, Y_{t-}, Z_t)dt + Z_t^* dM_t\\
 d\bar Y_t &= -\bar f(\omega, t, \bar Y_{t-}, \bar Z_t)dt + \bar Z_t^* dM_t
\end{align*}
for some processes $Z, \bar Z \in H^2_{M,\tloc}$ and some $\gamma$-balanced drivers $f$, $\bar f$. Then there exists a measure $\bQ$ such that 
\[Y_t - \bar Y_t + \int_{]0,t]} \big(f(\omega, s, Y_{s-}, Z_s) - \bar f(\omega, s, \bar Y_{s-}, Z_s)\big) ds\]
 is a $\bQ$-martingale. 
\end{lemma}
\begin{proof}
 For each $T$, let $\bQ^T$ be the measure given by Lemma \ref{lem:girsanovvalid}. Then for $t<T$, a simple rearrangment shows that 
\[\begin{split}&Y_t - \bar Y_t + \int_{]0,t]} \big(f(\omega, s, Y_{s-}, Z_s) - \bar f(\omega, s, \bar Y_{s-}, Z_s)\big) ds\\
   &=Y_0 - \bar Y_0 +\int_{]0,t]}\big(f(\omega, s, Y_{s-}, Z_s)- f(\omega,s,Y_{s-},\bar Z_s) \big)ds + \int_{]0,t]}(Z_s-\bar Z_s)^*dM_s
 \end{split}\]
and by Lemma \ref{lem:girsanovvalid} this is a $\bQ^T$-martingale up to time $T$. We now note that the measures $\bQ^T$ are consistent, in that $\bQ^T|_{\F_t} = \bQ^t|_{\F_t}$ for any $t\leq T$. As our underlying space is the space of paths of a countable state Markov chain, which can be embedded in the space of paths in $\bR$, Kolmogorov's extension theorem implies the existence of a measure $\bQ$ with $\bQ|_{\F_T}=\bQ^T|_{\F_T}$ for all $T$. 
\end{proof}

\begin{definition}\label{def:QfamDefn}
Let $\mathcal{Q}_\gamma$ denote the family of all measures $\bQ$ where $X$ has compensator $\lambda(\omega,t)$, for $\lambda$ a predictable process with $\bone^*\lambda\equiv 0$ and $(e_i^*\lambda(\omega, t))/(e_i^* A X_{t-}) \in [\gamma, \gamma^{-1}]$ for all $i$, where $0/0:=1$.
\end{definition}
\begin{remark}
Note that $\bP\in\mathcal{Q}_\gamma$, the measures which appear in Lemma \ref{lem:YisQmart} are in $\mathcal{Q}_\gamma$, and, for any $\bQ, \bQ'\in\mathcal{Q}_\gamma$, any stopping time $\tau$, the measure defined by $\tilde \bQ(A) = E^\bQ[E^{\bQ'}[I_A|\F_\tau]]$ is also in $\mathcal{Q}_\gamma$.
\end{remark}

It is worth noting that, in general, the measures in $\bQ_\gamma$ will be singular, even though their restrictions to $\F_T$ are absolutely continuous for every $T$. The following lemma gives us a slightly stronger result.

\begin{lemma} \label{lem:QtauequivPtau}
 Let $\tau$ be a stopping time with $\tau<\infty$ a.s., and let $\bQ\in\mathcal{Q}_\gamma$. Then $\bQ|_{\F_\tau}$ and $\bP|_{\F_\tau}$ are equivalent. 
\end{lemma}
\begin{proof}
As the relative rates of jumping are bounded, we know that $\bP|_{\F_T}$ and $\bQ|_{\F_T}$ are equivalent for any $T<\infty$. Let $A\in \F_\tau$. Then by definition of $\F_\tau$,
\[\bP(A\cap \{\tau<T\}) = 0 \Leftrightarrow \bQ(A\cap \{\tau<T\})=0\]
Now, as $\tau$ is almost surely finite valued, by the monotone convergence theorem
\[\lim_{T\to\infty} \bP(A\cap \{\tau<T\}) = \bP(A)\]
and similarly for $\bQ$. Hence
\[\begin{split} \bP(A) =0 &\Leftrightarrow \bP(A\cap \{\tau<T\}) =0 \quad \forall T\\ & \Leftrightarrow \bQ(A\cap \{\tau<T\}) =0 \quad \forall T \Leftrightarrow \bQ(A)=0  \end{split}
\]
and so $\bP$ and $\bQ$ are equivalent on $\F_\tau$.
\end{proof}

\section{BSDEs to Stopping times}\label{STBSDE}

In the following, we make great use of the following version of Markov's inequality. 
`For any stopping time $\tau$, any $T>0$, $E[I_{\tau>T}|\mathcal{F}_t] \leq E[\tau^{1+\beta}|\mathcal{F}_t]\,T^{-(1+\beta)}.$'

Our key result is as follows.
\begin{theorem}\label{thm:STBSDE}
 Suppose $\xi$ is $\F_\tau$-measurable and for some nondecreasing functions $K, \tilde K:\bR^+\to[1,\infty[$, some constants $\beta, \tilde\beta>0$,
\begin{align*}
E^{\bQ}[|\xi||\F_t]&\leq K(t),\qquad E^\bQ[(1+\tau)^{1+\beta}|\F_t]\leq K(t),\qquad E^\bQ[K(\tau)^{1+\tilde \beta}|\F_t]\leq \tilde K(t),
\end{align*}
all $\bP$-a.s. for all $\bQ\in\mathcal{Q}_\gamma$ and all $t$. Let $f:\Omega\times\bR^+\times\bR\times\bR^N \to \bR$ be $\gamma$-balanced, and such that for any $y, y', z$,
\begin{align*}
|f(\omega, t, 0, 0)| &\leq c(1+t^{\hat\beta}),\\
\frac{f(\omega, t, y, z)-f(\omega, t, y', z)}{y-y'} &\in [-c, 0]
\end{align*}
for some $c\in\bR$,  some $\hat\beta\in[0,\beta[$. Then the BSDE (\ref{eq:STBSDE}) admits a unique adapted solution satisfying the bound
\[|Y_t| \leq (1+c)K(t).\]
\end{theorem}

\begin{proof}
We first prove a solution exists. Let $\tau_n := \tau\wedge n$. Define 
\begin{equation}\label{eq:rdefn}
 r(\omega, u, y, y', z):= -\frac{f(\omega, u, y, z)-f(\omega, t, y', z)}{y-y'},
\end{equation}
and notice by assumption that $r\in [0,c]$ and is predictable. Define $(Y^n, Z^n)$ to be the solution to the finite-horizon BSDE
\begin{equation}\label{eq:nBSDE}
 Y^n_t = \xi I_{\tau\leq n}+\int_{]t,n]} f(\omega, u, Y^n_u,  Z^n_u)I_{u<\tau} du - \int_{]t,n]} (Z^n_u)^* dM_u.
\end{equation}
As $\xi$ is $\F_\tau$ measurable, we see that $Z^n_u=0$ whenever $u\geq \tau$. Now for $n\geq m$, by Lemma \ref{lem:girsanovvalid} and It\=o's Lemma, there exists a measure $\bQ$ (depending on $n,m$) such that

\begin{align*}
Y_m^n &= E^{\bQ}\Big[\xi I_{\tau\leq n}+ \int_{]\tau_m,\tau_n]} f(\omega, u, Y^n_{u-},  0)du\Big|\F_m\Big]\\
&= E^{\bQ}\Big[\exp\Big(-\int_{]\tau_m,\tau_n]}r(\omega, u, Y_{u-}^n, 0, 0) du\Big)\\
&\qquad \qquad \times \Big(\xi I_{\tau\leq n}+ \int_{]\tau_m,\tau_n]} f(\omega, u, 0,  0)du\Big)\Big|\F_m\Big]\\
 Y_t^n-Y_t^m &= E^{\bQ}\Big[(Y_m^n-\xi I_{\tau\leq m}) +\int_{]t,\tau_n]} f(\omega, u, Y^n_{u-}, Z_u^n) - f(\omega, u, Y^m_{u-}, Z_u^n)\Big|\F_t\Big]\\
&= E^{\bQ}\Big[\exp\Big(-\int_{]t,\tau_n]} r(\omega, u, Y_{u-}^n, Y_{u-}^m, Z_{u}^n) du\Big)(Y_m^n-\xi I_{\tau\leq m})\Big|\F_t\Big]
\end{align*}

Combining these, as $r\geq 0$ and by our assumptions $|\xi|\leq K(\tau)$, we see
\begin{align*}
|Y_t^n-Y_t^m| &\leq E^{\bQ}[|\xi I_{\tau\leq n} - \xi I_{\tau\leq m}|+ \int_{]\tau_m,\tau_n]} |f(\omega, u, 0,  0)|du |\F_t]|\\
&\leq E^{\bQ}[(K(\tau))I_{\tau>m} + c(1+\tau^{\hat\beta})(\tau_n-\tau_m)|\F_t]\\
&\leq E^{\bQ}[K(\tau)I_{\tau>m}|\F_t] + cE^{\bQ}[\tau(1+\tau^{\hat\beta}) I_{\tau>m}|\F_t]\\
&\leq E^{\bQ}[K(\tau)^{1+\tilde\beta}|\F_t]^{1/(1+\tilde\beta)}E^{\bQ}[I_{\tau>m}|\F_t]^{\tilde\beta/(1+\tilde\beta)}\\
&\qquad + cE^{\bQ}[\tau(1+\tau^{\hat\beta})|\F_t]^{1/(1+\alpha)}E^{\bQ}[I_{\tau>m}|\F_t]^{\alpha/(1+\alpha)}\\
&\leq \tilde K(t)^{1/(1+\tilde\beta)}\Big(\frac{K(t)}{m^{1+\beta}}\Big)^{\tilde\beta/(1+\tilde\beta)} + 2cK(t)^{1/\alpha}\Big(\frac{K(t)}{m^{1+\beta}}\Big)^{\alpha/(1+\alpha)}
\end{align*}
where $\alpha := \frac{\beta-\hat\beta}{1+\hat\beta}$. Hence, writing $\tilde\alpha:= (1+\beta) \Big(\frac{\tilde\beta}{1+\tilde\beta}\wedge \frac{\alpha}{1+\alpha}\Big)$, as we know $\tilde\alpha>0$, we have a bound of the form
\[|Y_t^n-Y_t^m| \leq \hat K(t) m^{-\tilde\alpha}\]
for some nondecreasing function $\hat K(t)$, and so $Y^n$ is a Cauchy sequence (converging uniformly in $\omega$ on compacts in time).

Now define $Y_t:=\lim_n Y^n_t$. It is easy to see that this satisfies the desired dynamics. Furthermore,  for some measure $\bQ$ we have the representation
\[Y_t = E^{\bQ}\Big[e^{-\int_{]t,\tau]} r(\omega, u, Y_{u-}, 0,0)du}\Big(\xi +\int_{]t,\tau]} f(\omega, u, 0, 0) du\Big) \Big| \F_t\Big]\]
from which, given our assumptions on $f$ and $\tau$, we can see that 
\[|Y_t| \leq (1+c)K(t).\]

Now suppose we have two solutions $Y$ and $\tilde Y$. Then from the assumed bound and the fact $Y_T = \tilde Y_T = \xi$ on the set $\tau\leq T$, we can see that for some measure $\bQ$ given by Lemma \ref{lem:girsanovvalid},
\begin{align*}
 E^{\bP}[|Y_t - \tilde Y_t|] &= E^{\bP}[|E^{\bQ}[e^{-\int_{]t, T\wedge\tau]} r(\omega, u, Y_{u-}, \tilde Y_{u-}, Z_{t}) du}(Y_T - \tilde Y_T)|\F_t]|] \\
&\leq E^{\bP}[E^{\bQ}[|Y_T - \tilde Y_T||\F_t]]\\
&\leq E^{\bP}[2(1+c)E^{\bQ}[K(T) I_{\tau>T}|\F_t]]\\
&\leq 2(1+c)E^{\bP}[E^{\bQ}[K(\tau)I_{\tau>T} |\F_t]]\\
&\leq 2(1+c)E^{\bP}[E^{\bQ}[K(\tau)^{1+\beta}|\F_t]]^{1/(1+\beta)} \, E^{\bP}[E^{\bQ}[I_{\tau>T}|\F_t]]^{\beta/(1+\beta)}\\
&\leq 2(1+c)\tilde K(0)^{1/(1+\beta)}\Big(\frac{K(t)}{T^{1+\beta}}\Big)^{\beta/(1+\beta)}\\
&\leq 2(1+c)\tilde K(0)^{1/(1+\beta)}K(t)^{\beta/(1+\beta)}T^{-\beta}.
\end{align*}
As $\beta>0$, letting $T\to\infty$ implies that $E^{\bP}[|Y_t - \tilde Y_t|]=0$ for all $t$, and so the solutions coincide almost surely. As the solutions are c\`adl\`ag, they agree up to indistinguishability.
\end{proof}
\begin{remark}
 We emphasise that the above proof permits the case when $f$ does not depend on $y$. In fact, the requirement that $f$ is monotone decreasing is not strictly necessary for our proof. It is enough to guarantee that the discounting terms $\int_{]s,t]} r(\omega, u, y, y', z) du$ are uniformly bounded above for all $s<t$ and all processes $y, y', z$,  as this just introduces a positive constant into our estimates, which will then also appear in the bound on $Y_t$.
\end{remark}

We now extend our existence result to drivers which are only `locally' $\gamma$-balanced and Lipschitz (in $y$), under slightly more restrictive assumptions on the terminal condition and the driver evaluated at zero.

\begin{theorem}\label{thm:STBSDEloc}
 Suppose $\xi$ is $\F_\tau$-measurable and for some constant $k$
\[|\xi|\leq k,\qquad |f(\omega, t, 0, 0)| \leq k,\qquad E^\bQ[(\tau-t)^+|\F_t]\leq k\]
and for some $\beta>0$, some nondecreasing functions $K, \tilde K:\bR\to [1, \infty[$
\[E^\bQ[(1+\tau)^{1+\beta}|\F_t]\leq K(t),\qquad E^\bQ[K(\tau)^{1+\beta}|\F_t]\leq \tilde K(t)\]
all $\bP$-a.s. for all $\bQ\in\mathcal{Q}_\gamma$ and all $t$.

For $y, n\in\bR$, let $y^{(n)}:=(-n)\vee y \wedge n$, and for $z\in\bR^N$ let $z^{(n)}$ be the vector with components $(e_i^*z)^{(n)}$. For $f:\Omega\times\bR^+\times\bR\times\bR^N \to \bR$ a predictable function, suppose that 
\[f^{(n)}(\omega, t, y, z):= f(\omega, t, y^{(n)}, (z- (X_{t-}^*z)\bone)^{(n)})\]
satisfies the requirements of Theorem \ref{thm:STBSDE} for every $n>0$. (That is, $f^{(n)}$ is $\gamma$-balanced in $z$ and Lipschitz decreasing in $y$, where $\gamma$ and the Lipschitz constant can depend on $n$). Then the BSDE (\ref{eq:STBSDE}) with driver $f$ admits a unique bounded solution.
\end{theorem}

\begin{proof}
 First note that as $f^{(n)}$ satisfies the requirements of Theorem \ref{thm:STBSDE} for any $n$, the BSDE with driver $f^{(n)}$ admits a unique solution $(Y^n Z^n)$. We know that this solution has a representation
\[Y^n_t = E^{\bQ}\Big[e^{-\int_{]t,\tau]} r(\omega, u, Y^n_{u-}, 0,0)du}\Big(\xi +\int_{]t,\tau]} f^{(n)}(\omega, u, 0, 0) du\Big) \Big| \F_t\Big]\]
for some $\bQ\in\mathcal{Q}_{\gamma^n}$, where $r$ is as in (\ref{eq:rdefn}), and so 
\[|Y_t^n|\leq E^{\bQ}\Big[|\xi| +\int_{]t,\tau]} |f^{(n)}(\omega, u, 0, 0)| du \Big| \F_t\Big]\leq k+kE^{\bQ}[(\tau-t)^+| \F_t]\leq k(1+k).\]
From the fact that $\Delta Y_t^n = (Z_t^n)^*\Delta X_t$ and $\Delta X_t=e_i-e_j$ for some $i,j$, and that jumps of $X$ are totally inaccessible and $Z$ is predictable, we see that $|e_i^*Z_t^n|\leq 2k(1+k)$ up to equivalence $\sim_M$. Hence we have a bound on $Y^{n}$, $Z^{n}$ independent of $n$. Therefore, provided $n\geq 2k(1+k)$, we see that $Y^n_t, Z_t^n$ is a solution for the BSDE (\ref{eq:STBSDE}) with driver $f$, as the truncation will have no effect.

Conversely, suppose we had two bounded solutions $(Y,Z), (Y', Z')$. Then we could set $n>k'$, where $k'$ is a bound on the solutions,  so that the truncation has no effect, and both processes would solve the BSDE with driver $f^{(n)}$. However this BSDE satisfies Theorem \ref{thm:STBSDE} and so admits unique solutions, hence we have a contradiction.
\end{proof}
\begin{remark}
 The peculiar definition of $f^{(n)}$ in the previous theorem (in particular the unnatural use of $(z- (X_{t-}^*z)\bone)^{(n)}$ rather than $z^{(n)}$, is simply to ensure that $f^{(n)}$ is Lipschitz continuous in the $\|\cdot\|_{M_t}$ seminorm whenever $f$ is, as the basic truncation $z\mapsto z^{(n)}$ is not invariant under $\|\cdot\|_{M_t}$ equivalence.
\end{remark}

The conditions of these statements may seem unusual and restrictive, however the following lemma gives a key example when they are satisfied. It is an immediate consequence of Corollary \ref{cor:hitmoments} in the next section, and so is stated without proof.

\begin{lemma}
 Let $\tau$ be the first hitting time of a set $\Xi\subseteq \X$. Let $\xi$ be a random variable of the form
\[\xi = g(\tau, X_\tau)\]
for some function $g$ with $g(t,x)\leq k (1+t^\beta)$ for some $k, \beta>0$. Then there exist functions $K, \tilde K$ satisfying the requirements of Theorem \ref{thm:STBSDE}.
\end{lemma}

\begin{remark}
 Theorem \ref{thm:STBSDE} (but not Theorem \ref{thm:STBSDEloc}) would work equally well for Brownian-Motion-Based BSDEs, with the corresponding definition of the family of measures $\mathcal{Q}_\gamma$.
\end{remark}

\begin{theorem}[Comparison theorem]\label{thm:compthm}
Let $(Y, Z)$ and $(Y', Z')$ be the solutions to two BSDEs with drivers $f, f'$ and terminal values $\xi, \xi'$ satisfying the conditions of Theorem \ref{thm:STBSDE} or \ref{thm:STBSDEloc}. Suppose $f(\omega, t, y, z)\geq f'(\omega, t, y, z)$ for all $(y,z)$, $dt\times d\mathbb{P}$-a.s. Then $\xi\geq \xi'$ a.s. implies $Y_t\geq Y'_t$ a.s. up to indistinguishability. Furthermore, $Y_t = Y'_t$ a.s.~on $A\in\F_t$ if and only if  $Y_s= Y'_s$ a.s.~on $A$ for each $s>t$ and $f(\omega, s, Y_{s-}, Z_s) = f'(\omega, s, Y_{s-}',  Z_s')$ on $A\times]t,\tau]$, $d\bP\times dt$-a.s.
\end{theorem}
\begin{proof}
From Lemma \ref{lem:YisQmart}, we know that there is a measure $\bQ$ such that for any stopping time ${\tau^*}\leq \tau$,
\[
 Y_t-Y'_t = E^{\bQ}\Big[\Big(Y_{\tau^*}-Y'_{\tau^*} +\int_{]t,\tau^*]} f(\omega, u, Y_{u-}, Z_u)-f'(\omega, u, Y'_{u-}, Z_u) du\Big) \Big| \F_t\Big].
\]
From Lemma \ref{lem:QtauequivPtau}, we know that the stated assumptions $\xi\geq \xi'$, $f\geq f'$ hold both $\bP$ and $\bQ$-a.s. Suppose we have a set $B\in\F_t$ such that $Y_t-Y'_t<0$ on $B$. Then define the stopping time
\[\tau^* = \inf\{s\geq t:Y_s-Y'_s\geq 0\}\leq \tau.\]
From this and the definition of $r$ in (\ref{eq:rdefn}), it is easy to deduce that 
\[I_B|Y_t-Y'_t| \leq E^{\bQ}\Big[-\int_{]t,\tau^*]} r(\omega, u, Y_{u-}, Y'_{u-}, Z_u) I_B|Y_{u-}-Y'_{u-}| du \Big| \F_t\Big].\]
However, as $r\geq 0$, the right hand side is nonpositive, so $I_B|Y_t-Y'_t| =0$, that is, $Y_t\geq Y'_t$ except on a null set.

Now, for any $A\in\F_t$, it is easy to verify from the representation
\[Y_t-Y'_t=E^{\bQ}\Big[\exp\Big(-\int_{]t, \tau]} r(\omega, u, Y_{u-}, Y_{u-}', Z_{u}) du\Big)(\xi-\xi')\Big|\F_t\Big]\]
 that $I_A(Y_t-Y'_t)=0$ if and only if $I_A(Y_s-Y'_s)=0$ a.s.~for each $s>t$ and thence $f(\omega, s, Y_{s-}, Z_s) = f'(\omega, s, Y_{s-}',  Z_s')$ on $A\times]t,\tau]$, $d\bP\times dt$-a.s.
\end{proof}

For completeness, we now state a useful result which gives us a solution of the form $Y_t=u(t, X_t)$. In \cite[Corollary 2]{Cohen2011b} this result is stated, but only for the case $N<\infty$ and $T$ deterministic.
\begin{theorem}\label{thm:Ymarkov}
Let $\tau$ be the first hitting time of a set $\Xi$, and let $Y$ be the solution to a BSDE with Markovian terminal condition $Y_\tau = \phi(\tau, X_\tau)$, for some bounded deterministic function $\phi:\bR^+\times \mathbb{R}^N\to\mathbb{R}$. Suppose $f$ satisfies the requirements of Theorem \ref{thm:STBSDE} or \ref{thm:STBSDEloc} and is Markovian in the sense that $f(\omega, t, y, z)= \tilde f(X_{t-}, t, y, z)$ for some $\tilde f$. Then there exists a measurablefunction $u:\bR^+\times\X\to\bR$ such that $Y_t = u(t,X_t)$ for all $t\leq\tau$. Furthermore, $u$ satisfies
\begin{itemize}
\item $u(t,x) = \phi(t, x)$ on $\bR^+\times \Xi$,
\item the associated vector $\mathbf{u}_t$ defined by $e_i^* \mathbf{u}_t = u(t,e_i)$ satisfies the $N$-dimensional ODE system
\[d\mathbf{u}_t = -(\mathbf{f}(t,\mathbf{u}_t) +A^* \mathbf{u}_t) dt \qquad \text{for }t<\tau\]
where $e_i^*\mathbf{f}(t,\mathbf{u}):= \tilde f(e_i, t, e_i^*\mathbf{u}, \mathbf{u})$, 
\item and the solution process $Z$ is given by $Z_t=\mathbf{u}_t$, up to equivalence $\sim_M$. In particular, note that the BSDE has a solution $Z$ which is deterministic.
\end{itemize}
\end{theorem}
\begin{proof}
 First note that, by the Markov property, we have
\begin{align*}
 Y_t &= E\Big[\phi(\tau, X_\tau) + \int_{]t,\tau]} f(X_{u-}, u, Y_{u-}, Z_{u}) du\Big|\F_t\Big]\\
& = E\Big[\phi(\tau, X_\tau) + \int_{]t,\tau]} f(X_{u-}, u, Y_{u-}, Z_{u}) du\Big|X_t\Big]\\
&= u(t, X_t)
\end{align*}
for some function $u$, where the third line follows from the Doob--Dynkin Lemma, which implies that for each $t$, as $Y_t$ is $\sigma(X_t)$-measurable, $Y_t$ is equal to a measurable function of $X_t$. Right-continuity of $Y$ in $t$ ensures that the function $u$ is measurable in the product space. Clearly $u(t,x) = \phi(t,x)$ on the set $t=\tau$, that is, on $\bR^+\times \Xi$. Applying \cite[Theorem 3.2]{Cohen2011b} (modified trivially to allow $N=\infty$ if necessary) we have the desired dynamics of $u$ and the statement $Z_t = \mathbf{u}_t$.
\end{proof}

We can also obtain a version of this result under which $Y$ does not depend on time.
\begin{theorem}\label{thm:Ytimhom}
Let $\tau$ be the first hitting time of a set $\Xi\subseteq\X$ and let $f:\X\times\bR^+\times\bR^N\to\bR$ satisfy the conditions of Theorem \ref{thm:STBSDEloc}. Consider the BSDE
\[Y_t = \phi(X_\tau) - \int_{]t,\tau]} f(X_{u-}, Y_{u-}, Z_u)du + \int_{]t,\tau]} Z_udM_u\]
where $\phi$ is a bounded function $\X\to\bR$. Then there exists a bounded function $u:\X\to\bR$ such that $Y_t = u(X_{t})$ and $e_i^*Z_t = u(e_i)$, so $Z$ is constant. The vector $\mathbf{u}$ defined by $e_i^*\mathbf{u}= u(e_i)$ satisfies the equation
\[f(x, x^*\mathbf{u}, \mathbf{u})  =-  \mathbf{u}^* Ax \qquad \text{for }x\in\X\setminus\Xi,\]
with boundary values $u(x)=\phi(x)$ for $x\in\Xi$.
\end{theorem}
\begin{proof}
As $\tau$ is a hitting time of a Markov chain $X$, we know
\[\begin{split}
Y_t &= E\Big[\phi(X_\tau) - \int_{]t,\tau]} f(X_{u-}, Y_{u-}, Z_u)du\Big|\F_t\Big]\\
& = E\Big[\phi(X_\tau) - \int_{]t,\tau]} f(X_{u-}, Y_{u-}, Z_u)du\Big|X_t\Big].
\end{split}
\]
Defining a new Markov chain $\hat X_s = X_{s-t}$ for $s>t$, with hitting time $\hat\tau=\tau-t$, by the Markov property we have
\[Y_t = E\Big[\phi(\hat X_{\hat\tau}) - \int_{]0,{\hat\tau}]} f(\hat X_{u-}, Y_{u-}, Z_u)du\Big|\hat X_0\Big]\]
and as the right hand side does not depend on $t$, we see that $Y_t$ is a function purely of $\hat X_0=X_t$.

Let $\hat A$ be the modification of the rate matrix such that $\hat A x= A x$ for $x\notin\Xi$ and $\hat A x = 0$ for $x\in\Xi$. This agrees with $A$ on the set $t<\tau$, so from Theorem \ref{thm:Ymarkov}, we see that there exists a vector $\mathbf{u}$ such that $Y_t= X_{t-}^*\mathbf{u}$, $Z_t=\mathbf{u}$,  and
\[0= -(\mathbf{f}(\mathbf{u}) +\hat A^* \mathbf{u}) dt \qquad \text{for }t<\tau.\]
However this equation does not depend on $t$, and so 
\[\mathbf{f}(\mathbf{u}) +\hat A^* \mathbf{u} = 0.\]
Premultiplying by $x\in\X$ and using the relation between $A$ and $\hat A$ and the boundary conditions, we obtain the desired equation. Boundedness of $u$ is trivial from this representation, as we have assumed $f$, $\phi$ are bounded.
\end{proof}

\section{Exponential hitting time bounds}\label{hitting}
We now seek to show that when $\tau$ is the hitting time of a state of the chain, the required $\tau$-integrability assumptions of Theorem \ref{thm:STBSDE} are satisfied. This is done by examining the exponential ergodicity of the Markov chain under perturbations of the rate matrix.

A key result of \cite{Cohen2012} is that, when the driver $f$ is $\gamma$-balanced, does not depend on time or $Y$, and depends on $\omega$ only through $X_t(\omega)$, the measures $\bQ$ and $\bP$ established by Lemma \ref{lem:YisQmart} have closely related ergodic properties. In particular, if the underlying Markov chain is uniformly ergodic under $\bP$, then it is also uniformly ergodic under $\bQ$, and the rate of convergence to the Ergodic distribution can be uniformly bounded for every $\bQ$ in terms of $\gamma$ and the properties of $\bP$.

This result is of fundamental importance to our approach to BSDEs up to stopping times. Our first step is to generalise away from the Markovian assumptions on $f$, and to work only with the assumption that $f$ is $\gamma$-balanced. In this case, notions of `ergodicity' cease to have a clear meaning, as the dynamics of the Markov chain under $\bQ$ can be time-dependent. Instead, we work with a more fundamental property, that of the existence of exponential moments of the first hitting times of states of the Markov chain.

We begin with the definition of uniform ergodicity.
\begin{definition}\label{def:unifergodic}
Let $\cal M$ denote the set of probability measures on $\X$, with the topology inherited from considering them as a convex subset of $\ell_1(\X)$ (the total variation topology, with norm $\|f\|_{TV} = \sum_x |f(x)|$). We write $P^A_t\mu$ for the law of $X_t$ given $X_0\sim \mu$ when $X$ evolves following the rate matrix $A$.

 We say the Markov chain $X$ is \emph{uniformly ergodic} if there exists a measure $\pi$ on $\X$, and constants $R,\rho>0$, such that
\[\sup_{\mu\in{\cal M}} \|P^A_t\mu-\pi\|_{TV}  \leq Re^{-\rho t} \qquad\text{for all }t.\]
In this case $\pi$ is the unique invariant measure for the chain.
\end{definition}
The following lemma is simply a variant of \cite[Theorem 16.2.2(iv)]{Meyn2009} and is stated without proof, but shall be useful in our understanding of this property.

\begin{lemma}\label{ergodlem1}
 Let $X$ be a uniformly ergodic Markov chain, and let $x_C$ be an arbitrary state. Let $\tau_{C}$ be the first hitting time of $x_C$. Then for some $\beta>0$ (and hence for all $\beta$ sufficiently small), $\sup_{x\in\X}E[e^{\beta \tau_C}|X_0=x]<\infty.$
\end{lemma}

We make the following basic assumption about our processes.

\begin{assumption}
 Under the measure $\bP$, the Markov chain $X$ has time-homogenous rate matrix $A$ and is uniformly ergodic.
\end{assumption}

\begin{theorem}\label{thm:hittingestimate}
Let $X,\bar X$ be two independent copies of the Markov chain on $\X$, and let these be uniformly ergodic under the measure generated by a rate matrix $A$. For some arbitrarily chosen state $\hat x$, let $\hat\tau=\inf\{t:X_t=\bar X_t = \hat x\}$, the first time $X$ and $\bar X$ meet in the state $\hat x$. Then for any $\epsilon>0$, there exists $\beta>0$ depending only on $\gamma$ and $A$ such that  
\[\sup_{B\succeq_\gamma A}\sup_{x,\bar x} E^B[e^{\beta \hat\tau}|X_0=x, \bar X_0=\bar x]\leq 1+\epsilon.\]
\end{theorem}
\begin{proof}
This is a re-expression of \cite[Lemma 11, Corollary 12]{Cohen2012}, where we note that the $T$ referred to here is not as stated in \cite{Cohen2012}, but is rather the more restrictive stopping time used in the proof of \cite[Lemma 11]{Cohen2012}.
\end{proof}

We can extend this result to give the following theorem.

\begin{theorem}\label{thm:exphtbounds}
For some arbitrarily chosen state $\hat x$, let $\hat\tau=\inf\{t:X_t= \hat x\}$. Then for any $\epsilon>0$ there exists a $\beta>0$ depending only on $A$ and $\gamma$ such that for any measure $\bQ\in\mathcal{Q}_\gamma$,
\[E^\bQ[e^{\beta (\hat\tau-t)^+}|\mathcal{F}_t]\leq (1+\epsilon)\]
for any $t$.
\end{theorem}
\begin{proof}
Define the BSDE driver 
\[g_\gamma(x,z):= \sup_{B:B\sim_\gamma A} (z^*(B-A)x).\]
By Lemma \ref{lem:gambalancedconvex} and Remark \ref{rem:measchangebalanced}, we can see that $g_\gamma$ is $\gamma$-balanced. By Lemma \ref{lem:balislip}, $g_\gamma$ is Lipschitz with respect to the $\|\cdot\|_M$ seminorm. 

Let $f(\omega, t, z)=z^*(\lambda_t -A X_{t-})$, where $\lambda$ is a vector process generating $\bQ\in\mathcal{Q}_\gamma$, in the sense of Definition \ref{def:QfamDefn}. For any $T>0$ we can consider the BSDEs with terminal value $e^{\beta (\hat\tau \wedge T)}$ at time $T$, and drivers $g_\gamma$ and $f$. Let $Y^g$ and $Y^{f}$ denote the corresponding solutions, for which direct calculation shows that
\[\begin{split}
 Y^g_t&= \sup_{B\sim_\gamma A} E^{B}[e^{\beta (\hat\tau \wedge T)}|\F_t]\\
   Y^{f}_t &= E^{\bQ}[e^{\beta (\hat\tau \wedge T)}|\F_t]
  \end{split}\]
We can see that $f$ is $\gamma$-balanced, and furthermore, that $g_\gamma(X_{t-},z)\geq f(\omega,t, z)$. By the finite-time comparison theorem (Theorem \ref{thm:finitecompthm}) we see that $Y^g_t\geq Y^{f}_t$ for all $t$. Hence
\[\sup_{B\sim_\gamma A} E^{B}[e^{\beta (\hat\tau \wedge T)}|\F_t]\geq E^{\bQ}[e^{\beta (\hat\tau \wedge T)}|\F_t]\]
and taking $T\to\infty$, as $\hat\tau$ is almost surely finite, by the monotone convergence theorem we see
\[E^{\bQ}[e^{\beta \hat\tau}|\F_t]\leq \sup_{B\sim_\gamma A} E^{B}[e^{\beta \hat\tau}|\F_t].\]
Now under each rate matrix $B$, 
\[E^{B}[e^{\beta \hat\tau}|\F_t] =  E^{B}[e^{\beta (t\wedge \hat\tau+(\hat\tau-t)^+)}|\F_t] \leq e^{\beta (t\wedge \hat\tau)} E^{B}[e^{\beta (\hat\tau-t)^+}|\F_t]
\]
For every $x\in\X$, by the Markov property and Theorem \ref{thm:hittingestimate}, 
\[E^{B}[e^{\beta (\hat\tau-t)^+}|X_{t}=x] = E^{B}[e^{\beta \hat\tau}|X_{0}=x] \leq 1+\epsilon.\]
Hence
\[E^{\bQ}[e^{\beta \hat\tau}|\F_t]=E^{\bQ}[e^{\beta \hat\tau}|X_t] \leq e^{\beta (t\wedge \hat\tau)}(1+\epsilon)\]
and rearrangement yields the result.
\end{proof}
\begin{corollary}\label{cor:hitmoments}
 Let $\hat\tau$ be the first hitting time of a state $\hat x$. Then for any $\beta\geq 0$, any $\bQ\in\mathcal{Q}_\gamma$, there exists a constant $k$ such that
\[E^\bQ[(1+\hat\tau)^{1+\beta}|\F_t] \leq k(1+t)^{1+\beta}\]
\end{corollary}
\begin{proof}
From Theorem \ref{thm:exphtbounds} we know that for all $\gamma$ sufficiently small, $E^\bQ[e^{\gamma (\hat\tau-t)^+}|\mathcal{F}_t]\leq (1+\epsilon)$. Hence, we can see that  for any $\beta>0$ there exists some constant $k$ such that
 \[E^\bQ[(1+(\hat\tau-t)^+)^{1+\beta}|\mathcal{F}_t]\leq k.\]
As $t\geq0$, we have the bound
\[\begin{split}(1+\hat\tau)^{1+\beta}&\leq (1+t)^{1+\beta}\Big(\frac{1+t+(\hat\tau-t)^+}{1+t}\Big)^{1+\beta}\\
&\leq (1+t)^{1+\beta}(1+(\hat\tau-t)^+)^{1+\beta}.
\end{split}\]
The result follows.
\end{proof}

\section{Applications}\label{applic}
We now present some novel applications of these methods. We begin with the archetypal Markovian control problem. A related approach for ergodic control problems was considered in \cite{Cohen2012}. A general setting for control of a Marked Point Process to deterministic times is given in \cite{Confortola2012}.

\subsection{Control to a stopping time}
Consider the problem of minimizing a cost
\[Y_t =  \essinf_{u} E^{u}\Big[\int_{]t,\tau]} L(s, Y_{s-}, X_{s-}, u_s) ds + \phi(\tau, X_{\tau})\Big|\F_t\Big]\]
where
\begin{itemize}
 \item $U$ is a space of controls, which is a separable metric space,
 \item $u$ is a $U$ valued predictable process,
 \item $\phi$ is a terminal cost function with $\phi(t, x)\leq c(1+t^\beta)$ for some $\beta>0$,
 \item $L:\bR^+\times \bR\times\X\times U\to\bR$ is a measurable cost function, with 
\[|L(t, y, x, u)| \leq c(1+t^{\hat\beta}),\qquad  (L(t, y, x, u)-L(t, y', x, u))/(y-y') \in [-c,0]\]
for some $c\in\bR^+$, some  $\hat\beta<\beta$,
 \item $E^{u}$ is the expectation under which at time $t$, for the path $\omega$, $X$ jumps from state $e_i$ to state $e_j$ at a rate $e_j^* A^{u_t(\omega)}e_i$, for some measurable matrix valued function $A^{(\cdot)}:U\to\text{rate matrices}$,
 \item for some $\gamma>0$, for all $u \in U$, the matrices $A^{u}\sim_\gamma A$, for some reference rate matrix $A$ under which $X$ is a uniformly ergodic Markov chain.
 \item $\tau$ is the first hitting time of a collection of points in $\X$.
\end{itemize}
We shall write $E$ for the expectation under which $X$ has rate matrix $A$.

We define the Hamiltonian
\begin{equation}\label{eq:hamiltonian}
 f(x, t,y, z) = \inf_{\upsilon\in U}\{L(t, y, x, \upsilon) + z^* (A^{\upsilon}-A)x\}.
\end{equation}
As $A^u\sim_\gamma A$, by Lemma \ref{lem:gambalancedconvex} we see that $f$ is $\gamma$-balanced and the requirements of Theorem \ref{thm:STBSDE} are satisfied. Therefore, the BSDE (\ref{eq:STBSDE}) with driver $f$ admits a unique solution with bounded growth. By Theorem \ref{thm:Ymarkov}, the solution to the BSDE is Markovian, that is, $Y_t = u(t, X_t)$ and $e_i^*Z_t = u(t, e_i)$, for some function $u$.

If the infimum in (\ref{eq:hamiltonian}) is attained, then there exists (assuming the continuum hypothesis, by McShane and Warfield \cite{McShane1967}) a measurable function $\kappa:\X\times \bR^+\times\bR\times\bR^N\to U$ such that
\[f(x,t,y,z) = L(t,y,x,\kappa(x, t,y,z)) + z^*(A^{\kappa(x, t,y,z)}-A)x.\]
 We then have the following theorem.

\begin{theorem}\label{thm:optimalcontrol}
 In the setting described above, let $(Y, Z)$ be the solution to the BSDE (\ref{eq:STBSDE}) with driver $f$. Then the following hold.
\begin{enumerate}[(i)]
 \item For an arbitrary control $u$, if
\[Y_t^u =  E^{u}_x\Big[\int_{]t,\tau]} L(s, Y_{s-}^u, X_{s-}, u_s) ds + \phi(\tau, X_{\tau})\Big|\F_t\Big]\]
then $Y_t^u\geq Y_t$, and equality holds if and only if
\[L(t, Y_{t-}, X_{t-}, u_t) + (Z_t)^* (A^{U_t}-A)X_{t-}= f(X_{t-},t, Y_{t-}, Z_t)\qquad d\bP\times dt-a.e.\]
\item If the infimum is attained in (\ref{eq:hamiltonian}), then the control $u_t^* = \kappa(X_{t-}, t, Y_{t-}, Z_t)$ verifies $Y^{u^*}_t= Y_t$, and is an optimal feedback control.
\end{enumerate}
\end{theorem}
\begin{proof}
We see that $Y_\tau = Y^u_\tau = \phi(\tau, X_\tau)$. By definition of the Hamiltonian, $f(x, t,y, z) \leq L(t, y, x, U_t) + z^* (A^U-A)x$, and so by the Comparison theorem (Theorem \ref{thm:compthm}), we see $Y_t^u \geq Y_t$ for every $t$, and the desired condition for equality holds. Statement (ii) also holds by an application of the Comparison theorem.
\end{proof}

\subsection{Stochastic shortest paths}
Let $\G$ be a directed graph, represented by nodes in $\X$. Let the distance from node $e_i$ to $e_j$ be given by $e_j^* D e_i$ for some matrix $D$. Let $A$ be the transition matrix of a continuous time random walk on this graph, so that $e_j^* A e_i = \frac{1/(e_i^*De_j)}{\sum_j 1/(e_i^*De_j)}$ for $i\neq j$, and $e_i^* A e_i = -\sum_{j\neq i} e_j^* A e_i$. The time to reach a node $x$ starting from a node $x'$ is then given by $E[\tau|X_0=x']$, where $\tau$ is the first hitting time of $x$. 

Now suppose that it is possible to choose to `walk faster' on those paths which are heading in the right direction. This could be modelled by a change of measure, where the transition matrix $A$ is replaced by a matrix $A^u$, where $u$ is some control. Provided $A^u\sim_\gamma A$ for some $\gamma$, as shown in the previous section, the expected time to hitting $x$ under the optimal control is then the solution of the BSDE to $\tau$ with driver $f(x, z) = \inf_{u}\{z^*(A^u-A)x\}$ and terminal value $\xi=\tau$.

Alternatively, we can consider the expected \emph{remaining} time to hitting $x$ under the optimal control as the solution of the BSDE to $\tau$ with driver $f'(x,z) = \inf_u\{z^*(A^u-A)x\}+1$, and terminal value $\xi'=0$. If the earlier BSDE has solution $(Y,Z)$, and this variant has solution $(Y',Z')$, then we can see that $Y_t=Y'_t+t$, but that $Y'_t$ satisfies the requirements of Theorem \ref{thm:Ytimhom}, and therefore its solution is of the form $Y'_t=u(X_t)$ for some $u:\X\to\bR$, which may be convenient for calculation.

\subsection{Reliability for networks with control}

Consider a model for transmission of messages over a finite network. A message is to be transmitted from a node $x_0$ to a node $x_1$, and each node $e_i$ naturally passes messages to node $e_j$ at a rate $e_j^*Ae_i$. We wish to examine the probability that the message eventually reaches its destination. When the network is flawless this probability is always one. However, suppose that each node $x$ loses the message at a rate $r_x$, and that there is a (possibly empty) collection of nodes $\Xi$ at which a message is irretrievably lost. Without any control, we can consider this as a linear BSDE to a stopping time. 

Consider the Markov chain describing the motion of a message. Let $\tau$ be the first hitting time of $\{x_1\}\cup\Xi$, and let $Y_t$ be the solution to the BSDE
\[I_{\{X_\tau=x_1\}}= Y_t - \int_{]t,\tau]} -r_{X_{t-}} Y_{u-} du + \int_{]t,\tau]} Z_u^*dM_u\]
so that 
\[Y_t = E[e^{-\int_{]t,\tau]}r_{X_{s-}}ds} I_{\{X_\tau=x_1\}}|\F_t].\]

Now suppose we can extend this model, so that each node has some control over where a message is sent. Each node has a control $u$ with which it can modify transitions so that they occur at a rate $e_i^*A^ue_j$ for some $A^u\sim_\gamma A$, so as to maximise the probability that the message reaches its target. The maximal probability is then given by solving the BSDE
\[I_{\{X_\tau=x_1\}}= Y_t - \int_{]t,\tau]} -r_{X_{s-}} Y_{s-} +\sup_u\{Z_s^*(A^u-A)X_{s-}\}ds + \int_{]t,\tau]} Z_u^*dM_u,\]
and the optimal policy for the active node $X_{t-}$ at time $t$ is given by
\[u^* = \argmax_u \{Z_t^*(A^u-A)X_{t-}\}.\]

\subsection{Non-Ohmic Electronic circuits}
We now give a different situation, where we consider an electronic circuit. The theory of circuits of resistors is described using Kirchoff's laws and Ohm's law, a more detailed presentation of the approach we take here can be found in \cite[Chapter 1]{Grimmett2010}. 

Consider a circuit of resistors. Let the circuit be described by a graph represented in $\X$, and let the edge $(e_i, e_j)$ have resistance $R_{i,j}$. Then it follows from Ohm's law and Kirchoff's laws that, on a source set $\Xi$, the voltage potential is a prescribed function $\phi$, and off the source set the voltage potential $v$ in the circuit is a harmonic function.  In particular, if $w_{i,j}= 1/ R_{i,j}$ denotes the conductance over the edge $(e_i, e_j)$, we have
\[\begin{cases}
   v(e_i)=\phi(e_i), & e_i\in\Xi,\\
v(e_i)\sum_j w_{i,j} = \sum_j w_{i,j} v(e_j),& e_i\notin \Xi.
  \end{cases}\]

Let $A$ be the matrix defined by $e_j^*A e_i = w_{i,j}$ for $i\neq j$ and $e_i^*A e_i = -\sum_j w_{i,j}$. Then we consider a Markov chain $X$ with rate matrix $A$, and can show that
\[v(e_i) = E^A[\phi(X_\tau)|X_0=e_i]\]
where $\tau$ is the first hitting time of the source set.

Now suppose that our circuit no longer consists purely of resistors, but that it also has diodes. These ubiquitous electronic components fail to satisfy Ohm's law, and so the above representation in terms of a Markov chain fails. However, it is possible to write down a nonlinear relationship between voltage, current and resistance which is satisfied (see for example, \cite[Section 5.2]{Crecraft2002}). In particular, if we know the voltage drop over an edge  $V_{i,j}= v(e_i)-v(e_j)$ we can write $I_{i,j}R_{i,j}(V_{i,j}) = V_{i,j}$, where $I_{i,j}$ is the current passing over the edge. From \cite{Crecraft2002}, the relation for an \emph{np}-type diode, for example, is given by 
\[I_{i.j}=I^s(\exp(V_{i,j}/V^T)-1)\]
where $I^s, V^T$ are constants based on the properties of the diode. Rearranging this gives the implied resistance
\[R_{i,j}(V_{i,j}) = \frac{V_{i,j}}{I^s(\exp(V_{i,j}/V^T)-1)}> 0\]
which is a Lipschitz function of $V_{i,j}$, and is bounded away from zero (and $\infty$) over any finite interval.

Writing $\mathbf{v}$ for the vector with entries $e_i^*\mathbf{v} = v(e_i)$, and $w_{i,j}(v) = (R_{i,j}(V_{ij}))^{-1}$ for the implied conductance, we can construct a matrix $A^\mathbf{v}$ with $e_j^*A^\mathbf{v} e_i = w_{i,j}$ for $i\neq j$ and $e_i^*A^\mathbf{v} e_i = -\sum_j w_{i,j}$. For a given voltage potential vector $\mathbf{v}$, this $A^\mathbf{v}$ is the transition matrix based on the implied conductances at the potential $\mathbf{v}$. 

Let $A$ be as before, now defined by replacing each diode with a resistor with resistances $\tilde R_{i,j}$. As the circuit is finite, $A$ generates a uniformly ergodic Markov chain. 

 We then consider the BSDE
\[\phi(X_\tau) = Y_t - \int_{]t,\tau]} Z_s^*(A^{Z_s}-A)X_{s-} ds + \int_{]t,\tau]} Z_t dM_t\]
This is a time homogenous Markovian BSDE to a first hitting time. The terminal value $\phi(X_\tau)$ is bounded as for each $n>0$, we can find $\gamma_n>0$ such that $A^{\mathbf{v}} \sim_{\gamma_n} A$ whenever $\|v\|\leq n$, we see that the driver $f(X_{s-}, Z_s) = Z_s^*(A^{Z_s}-A)X_{s-}$ is locally $\gamma$-balanced, and $f(x,0)=0$. By Theorem \ref{thm:STBSDEloc}, this BSDE admits a unique bounded solution, and by Theorem \ref{thm:Ytimhom} it is of the form $Y_t = v(X_t)$, which is precisely the voltage potential on the circuit with a diode.

Using the same methodology, one can also extend this method to other electronic components, for example, transistors.

\section{Conclusions}
We have considered BSDEs on Markov chains when the terminal time is replaced by a non-bounded stopping time, and no strict montonicity property in $Y$ (i.e. a `discounting' term) is given for $f$. We have shown that, given appropriate polynomial growth bounds on the terminal value and the driver, and given sufficient integrability of the stopping time, these equations admit unique solutions with uniformly controlled growth in time. For the special case of first hitting times of a set, we have seen that the integrability conditions on the stopping time are indeed satisifed, and so these BSDEs admit unique solutions.

We have then considered various applications to these equations, in particular to control problems which terminate at a hitting time, and to problems of nonlinear behaviour on certain graphs, for example to network reliability. We have also seen how this gives a stochastic approach to our understanding of electronic circuits without Ohm's law.

We can also extend these approaches to more general types of BSDEs, either to the classic Brownian setting of \cite{Pardoux1990}, or to allow for Poisson jumps as in \cite{Royer2006}. For the Brownian case, the required ergodicity properties may be derivable from \cite{Debussche2011}, however such an extension is non-trivial. On the other hand, if $\tau$ is restricted to be the first exit time of a set, then standard estimates may be available to give the required integrability, in which case our results will extend directly.

\bibliographystyle{plain}
\bibliography{../../RiskPapers/General}
\end{document}